\documentclass[12pt]{article}
\usepackage{amsmath, amsthm, amscd, amsfonts, amssymb, graphicx}%, color}
\usepackage{titles}
\usepackage{algorithm}
\usepackage{algpseudocode}
\usepackage{subfigure} 
 \usepackage{caption}
 \usepackage{graphics,graphicx}

\usepackage{epsfig}
\usepackage{color}
\usepackage{fancyhdr}
\usepackage{titlesec}
\usepackage{tikz}
\usepackage{afterpage}
\usetikzlibrary{arrows}
\usetikzlibrary{decorations.markings}
\tikzstyle{block}=[draw opacity=0.7,line width=1.4cm]
\tikzset{
  big black arrow/.style={
    decoration={markings,mark=at position 1 with {\arrow[scale=2.5,black]{>}}},
    postaction={decorate},
    shorten >=0.4pt},
    line/.style={draw, ->}}

\setbox0=\hbox{$+$}
\newdimen\plusheight
\plusheight=\ht0
\def\+{\;\lower\plusheight\hbox{$+$}\;}

\setbox0=\hbox{$-$}
\newdimen\minusheight
\minusheight=\ht0
\def\-{\;\lower\minusheight\hbox{$-$}\;}

\setbox0=\hbox{$\cdots$}
\newdimen\cdotsheight
\cdotsheight=\plusheight%\ht0
\def\cds{\lower\cdotsheight\hbox{$\cdots$}}

\renewcommand{\(}{\left\(}
\renewcommand{\)}{\right\)}

\numberwithin{equation}{section}
\theoremstyle{plain}
\newtheorem{theorem}{Theorem}[section]

\newtheorem{corollary}[theorem]{Corollary}

\newtheorem{remarks}[theorem]{Remark}

\usepackage{pdfpages}

\begin{document}
\begin{center}{\bf Some restricted partition functions in terms of $2$-adic valuation}\end{center}\begin{center}	
		\footnotesize{\bf Sabi Biswas and Nipen Saikia$^{\ast}$}\\
						Department of Mathematics, Rajiv Gandhi
			University,\\ Rono Hills, Doimukh-791112, Arunachal Pradesh, India.\\
				E. Mail(s): sabi.biswas@rgu.ac.in; nipennak@yahoo.com\\
				$^\ast$\textit{Corresponding author}.\end{center}\vskip2mm
				
				\noindent {\bf Abstract:} The $2$-adic valuation of an integer $n$ which is the exponent of the highest power of $2$ that divides $n$. In this paper, we give representations of certain restricted partition functions in terms of $2$-adic valuation.  
	
\section{Introduction}
For $n\geq1$, partition of a positive integer $n$ is defined as a finite sequence of positive integers $(\omega_1, \omega_2, ..., \omega_k)$ such that $\sum_{j=1}^k\omega_j=n;\quad \omega_j\ge \omega_{j+1},$ where $\omega_j$ are called parts or summands of the partition. Let $p(n)$ denote the number of partitions of $n$ with $p(0)=1$. For example,  $p(3)=3$  and the partitions are 
$3,\quad  2+1$ and $1+1+1.$ 

If the order of integers $\omega_j$ does not matter, then $\sum_{j=1}^k\omega_j=n$ is known as an integer partition and can be rewritten as $$n=t_1+2t_2+...+nt_n,$$ where each positive integer $j$ appears $t_j$ times in the partition. Then the number of parts of this partition is given by $$t_1+t_2+...+t_n=k.$$For example,
\begin{align*}3
&=0\cdot 1+0\cdot 2+1\cdot 3\\
&=1\cdot 1+1\cdot 2+0\cdot 3\\
&=3\cdot 1+0\cdot 2+0\cdot 3.
\end{align*}
The generating function for the partition function $p(n)$ \cite{LE} is given by
\begin{equation*}
\sum_{n=0}^{\infty}p(n)q^{n} = \frac{1}{(q;q)_{\infty}},
\end{equation*}where for any complex numbers $L$ and $q$ with $|q|<1$, $(L;q)_{\infty}$ can be defined as
\begin{equation*}
(L;q)_{\infty} = \prod_{s=0}^{\infty}(1-Lq^s),\quad \text{where} \quad  (L;q)_{0} = 1\quad \text{and}\quad (L;q)_{n} =\prod_{s=0}^{n-1}(1-Lq^s), \quad  n\ge1.
\end{equation*}
The $2$-adic valuation of an integer $n$ is the highest power of $2$ dividing $n$. Let $v_2(n)$ denote the $2$-adic valuation of an integer $n$, then 
\begin{equation*}
v_2(n):=max\{k\in \mathbb{N}_0:2^k| n\},\quad v_2(0)=\infty,
\end{equation*}where $\mathbb{N}_0$ denotes the set of natural numbers (including zero) and $2^k| n$ denotes divisibility of $n$ by $2^k$.
Recently, Merca~\cite{M, M1} established some results on Euler's partition function $p(n)$ and overpartition function $\overline{p}(n)$ in terms of $v_2(n)$, where
\begin{equation*}
\sum_{n=0}^{\infty}\overline p(n)q^{n} = \frac{(-q;q)_\infty}{(q;q)_{\infty}}.
\end{equation*} 
 In this paper, we prove that some restricted partition functions can be expressed in terms of $2$-adic valuation. Specifically, we give  representations of  overpartitions into odd parts, partitions where even parts are distinct and odd parts may be repeated, partitions into distinct parts, partitions into distinct odd parts where even parts may be repeated and partitions into even parts in terms of $v_2(n)$.

\section{Representations of partition functions in terms of $v_2(n)$ }
In this section, we derive alternative representations for the generating functions of some restricted partition functions with the help of $2$-adic valuation. 

\begin{theorem}\label{th1}Let for any positive integer $n$ and $|q|<1$, $\bar{p}_o(n)$ denote  the number of overpartitions of an positive integer $n$ into odd parts with the generating function given by
	$$ \sum_{n=0}^{\infty}\bar{p}_o(n)q^n=\dfrac{(-q;q^2)_\infty}{(q;q^2)_\infty},\quad \bar{p}_o(0)=1.$$
Then  
\begin{equation}\label{1}
\sum_{n=0}^{\infty}\bar{p}_o(n)q^n=\prod_{n=1}^{\infty}(1+q^n)^{v(n)},
\end{equation}
 where $$ v(n)=\left\{
 \begin{array}{cc}
1+v_2(4n-2), \quad
  if~ n=2k-1, k=1,2,3..., \\
 v_2(4n-2), \quad if~ n=2k, k=1,2,3....
 \end{array}
 \right.$$
\end{theorem}
\begin{proof}
Consider the identity
\begin{equation}\label{o1}
\dfrac{1}{(1-q)}=\prod_{k=0}^{\infty}(1+q^{2^k}),\quad |q|<1.
\end{equation}Replacing $q$ by $q^{2n-1}$ in \eqref{o1}, we obtain
\begin{equation}\label{o2}
\dfrac{1}{(1-q^{2n-1})}=\prod_{k=0}^{\infty}(1+q^{2^k\cdot (2n-1)}).
\end{equation}Now, the generating function of $\bar{p}_o(n)$ can be rewritten as
\begin{align*}\sum_{n=0}^{\infty}\bar{p}_o(n)q^n
&=\prod_{n=1}^{\infty}\dfrac{1+q^{2n-1}}{1-q^{2n-1}}\notag\\
&=\prod_{n=1}^{\infty}(1+q^{2n-1})\prod_{k=0}^{\infty}(1+q^{2^k\cdot (2n-1)})\notag\\
&=\prod_{n=1}^{\infty}(1+q^n)^{v(n)}.
\end{align*} \end{proof}

\begin{corollary}\label{th2}
Let $n$ be a positive integer, then we have
\begin{equation*}
\bar{p}_o(n)=\sum_{\substack{t_1+2t_2+...+nt_n=n\\ t_k\leq v(k)}}^{}\binom{v(1)}{t_1}\binom{v(2)}{t_2}...\binom{v(n)}{t_n},
\end{equation*} where $v(n)$ is defined as in equation \eqref{1} of Theorem~\ref{th1}.
\end{corollary}
\begin{proof}By employing binomial expansion, equation \eqref{1} of Theorem~\ref{th1} can be rewritten as
\begin{equation}\label{o4}
\sum_{n=0}^{\infty}\bar{p}_o(n)q^n=\prod_{n=1}^{\infty}\left(\sum_{j=0}^{v(n)}\binom{v(n)}{j}q^{j\cdot n}\right).
\end{equation}By applying Cauchy multiplication of power series, the product of the right hand side of \eqref{o4} can be written as
\begin{equation}\label{o5}
\sum_{n=0}^{\infty}\bar{p}_o(n)q^n=\sum_{n=0}^{\infty}q^n\sum_{t_1+2t_2+...+nt_n=n}^{}\prod_{j=1}^{n}\binom{v(j)}{t_j}.
\end{equation}Hence, the proof easily follows from \eqref{o5} on comparing the coefficients of $q^n$.
\end{proof}

\begin{remarks}
The summation on the right hand side of Corollary~\ref{th2} contains all partitions of $n$ except the terms that contains $t_k>v(k)$ as the  binomial coefficient $\binom{v(k)}{t_k}$ becomes zero for $t_k>v(k)$. As an illustration, the partitions of $5$ in which each part $k$ has the multiplicity at most $v(k)$ are
$$5,\quad 4+1,\quad 3+2,\quad 3+1+1.$$
Therefore, by Corollary~\ref{th2}, we have
\begin{align*}
\bar{p}_o(5)
&=\binom{v(5)}{1}+\binom{v(4)}{1}\binom{v(1)}{1}+\binom{v(3)}{1}\binom{v(2)}{1}+\binom{v(3)}{1}\binom{v(1)}{2}\\
&=\binom{2}{1}+\binom{1}{1}\binom{2}{1}+\binom{2}{1}\binom{1}{1}+\binom{2}{1}\binom{2}{2}\\
&=2+2+2+2=8.
\end{align*}\end{remarks}

\begin{theorem}\label{th3}Let for any positive integer $n$ and $|q|<1$, $ped(n)$ denote the number of partitions of an positive integer $n$ into distinct even parts where odd parts may be repeated and defined by
$$\sum_{n=0}^{\infty} ped(n)q^n=\dfrac{(-q^2;q^2)_\infty}{(q;q^2)_\infty},\quad ped(0)=1.$$ Then
\begin{equation}\label{2}
\sum_{n=0}^{\infty}ped(n)q^n=\prod_{n=1}^{\infty}(1+q^n)^{v(n)},
\end{equation}
 where $$ v(n)=\left\{
 \begin{array}{cc}
1+v_2(4n-2), \quad
  if~ n=2k, k=1,2,3..., \\
 v_2(4n-2), \quad if~ n=2k-1, k=1,2,3....
 \end{array}
 \right.$$
\end{theorem}
\begin{proof}
The generating function of $ped(n)$ can be rewritten as
\begin{align*}\sum_{n=0}^{\infty}ped(n)q^n
&=\prod_{n=1}^{\infty}\dfrac{1+q^{2n}}{1-q^{2n-1}}\notag\\
&=\prod_{n=1}^{\infty}(1+q^{2n})\prod_{k=0}^{\infty}(1+q^{2^k\cdot (2n-1)})=\prod_{n=1}^{\infty}(1+q^n)^{v(n)}.
\end{align*}
\end{proof}

\begin{corollary}\label{thm1}
Let $n$ be a positive integer, then we have
\begin{equation*}
ped(n)=\sum_{\substack{t_1+2t_2+...+nt_n=n\\ t_k\leq v(k)}}^{}\binom{v(1)}{t_1}\binom{v(2)}{t_2}...\binom{v(n)}{t_n},
\end{equation*} where $v(n)$ is defined as in equation \eqref{2} of Theorem~\ref{th3}.
\end{corollary}
\begin{proof}
Proceeding in the same way as in Corollary~\ref{th2}, we complete the proof.
\end{proof}

\begin{remarks}
 According to Corollary~\ref{thm1}, the partitions of $5$ in which each part $k$ has the multiplicity at most $v(k)$ are
$$5,\quad 4+1,\quad 3+2,\quad 2+2+1.$$
\begin{align*}
ped(5)
&=\binom{v(5)}{1}+\binom{v(4)}{1}\binom{v(1)}{1}+\binom{v(3)}{1}\binom{v(2)}{1}+\binom{v(2)}{2}\binom{v(1)}{1}\\
&=\binom{1}{1}+\binom{2}{1}\binom{1}{1}+\binom{1}{1}\binom{2}{1}+\binom{2}{2}\binom{1}{1}=6.
\end{align*}\end{remarks}

\begin{theorem}\label{th4}Let for any positive integer $n$ and $|q|<1$, $p_d(n)$ denote the number of partitions of an positive integer $n$ into distinct parts and defined by
$$\sum_{n=0}^{\infty} p_d(n)q^n=\dfrac{1}{(q;q^2)_\infty},\quad p_d(0)=1.$$
Then
\begin{equation}\label{3}
\sum_{n=0}^{\infty}p_d(n)q^n=\prod_{n=1}^{\infty}(1+q^n)^{v_2(4n-2)}.
\end{equation}
\end{theorem}
\begin{proof}
The generating function of $p_d(n)$ can be rewritten as
\begin{align*}\sum_{n=0}^{\infty}p_d(n)q^n
&=\prod_{n=1}^{\infty}\dfrac{1}{1-q^{2n-1}}\notag\\
&=\prod_{n=1}^{\infty}\prod_{k=0}^{\infty}(1+q^{2^k\cdot (2n-1)})=\prod_{n=1}^{\infty}(1+q^n)^{v_2(4n-2)}.
\end{align*} 
\end{proof}

\begin{corollary}\label{thm2}
Let $n$ be a positive integer, then we have
\begin{equation*}
p_d(n)=\sum_{\substack{t_1+2t_2+...+nt_n=n\\ t_k\leq v_2(4k-2)}}^{}\binom{v_2(2)}{t_1}\binom{v_2(6)}{t_2}...\binom{v_2(4n-2)}{t_n}.
\end{equation*}
\end{corollary}
\begin{proof}
Proceeding in the same way as in Corollary~\ref{th2}, we complete the proof.
\end{proof}

\begin{remarks}
 According to Corollary~\ref{thm2}, the partitions of $5$ in which each part $k$ has the multiplicity at most $v_2(4k-2)$ are
$$5,\quad 4+1,\quad 3+2.$$
\begin{align*}
p_d(5)
&=\binom{v_2(18)}{1}+\binom{v_2(14)}{1}\binom{v_2(2)}{1}+\binom{v_2(10)}{1}\binom{v_2(6)}{1}\\
&=\binom{1}{1}+\binom{1}{1}\binom{1}{1}+\binom{1}{1}\binom{1}{1}=3.
\end{align*}\end{remarks}

\begin{theorem}\label{th5}Let for any positive integer $n$ and $|q|<1$, $pod(n)$ denote the number of partitions of an positive integer $n$ into distinct odd parts where even parts may be repeated and defined by
$$\sum_{n=0}^{\infty} pod(n)q^n=\dfrac{(-q;q^2)_\infty}{(q^2;q^2)_\infty},\quad pod(0)=1.$$
Then
\begin{equation}\label{4}
\sum_{n=0}^{\infty}pod(n)q^n=\prod_{n=1}^{\infty}(1+q^n)^{v(n)},
\end{equation}
 where $$ v(n)=\left\{
 \begin{array}{cc}
v_2(n), \quad
  if~ n=2k, k=1,2,3..., \\
 v_2(4n-2), \quad if~ n=2k-1, k=1,2,3....
 \end{array}
 \right.$$
\end{theorem}
\begin{proof}
The generating function of $pod(n)$ can be rewritten as
\begin{align*}\sum_{n=0}^{\infty}pod(n)q^n
&=\prod_{n=1}^{\infty}\dfrac{1+q^{2n-1}}{1-q^{2n}}\notag\\
&=\prod_{n=1}^{\infty}(1+q^{2n-1})\prod_{k=0}^{\infty}(1+q^{2^k\cdot 2n})=\prod_{n=1}^{\infty}(1+q^n)^{v(n)}.
\end{align*}
\end{proof}

\begin{corollary}\label{thm3}
Let $n$ be a positive integer, then we have
\begin{equation*}
pod(n)=\sum_{\substack{t_1+2t_2+...+nt_n=n\\ t_k\leq v(k)}}^{}\binom{v(1)}{t_1}\binom{v(2)}{t_2}...\binom{v(n)}{t_n},
\end{equation*} where $v(n)$ is defined as in equation \eqref{4} of Theorem~\ref{th5}.
\end{corollary}
\begin{proof}
Proceeding in the same way as in Corollary~\ref{th2}, we complete the proof.
\end{proof}

\begin{remarks}
According to Corollary~\ref{thm3}, the partitions of $5$ in which each part $k$ has the multiplicity at most $v(k)$ are
$$5,\quad 4+1,\quad 3+2.$$
\begin{align*}
pod(5)
&=\binom{v(5)}{1}+\binom{v(4)}{1}\binom{v(1)}{1}+\binom{v(3)}{1}\binom{v(2)}{1}\\
&=\binom{1}{1}+\binom{2}{1}\binom{1}{1}+\binom{1}{1}\binom{1}{1}=4.
\end{align*}\end{remarks}

\begin{theorem}\label{th6}Let for any positive integer $n$ and $|q|<1$, $p_e(n)$ denote the number of partitions of an positive integer $n$ into even parts with the generating function given by
$$\sum_{n=0}^{\infty} p_e(n)q^n=\dfrac{1}{(q^2;q^2)_\infty},\quad p_e(0)=1.$$
Then
\begin{equation}\label{5}
\sum_{n=0}^{\infty}p_e(n)q^n=\prod_{n=1}^{\infty}(1+q^n)^{v_2(n)}.
\end{equation}
\end{theorem}
\begin{proof}
The generating function of $p_e(n)$ can be rewritten as
\begin{align*}\sum_{n=0}^{\infty}p_e(n)q^n
&=\prod_{n=1}^{\infty}\dfrac{1}{1-q^{2n}}\notag\\
&=\prod_{n=1}^{\infty}\prod_{k=0}^{\infty}(1+q^{2^k\cdot (2n)})=\prod_{n=1}^{\infty}(1+q^n)^{v_2(n)}.
\end{align*} 
\end{proof}

\begin{corollary}\label{thm4}
Let $n$ be a positive integer, then we have
\begin{equation*}
p_e(n)=\sum_{\substack{t_1+2t_2+...+nt_n=n\\ t_k\leq v_2(k)}}^{}\binom{v_2(1)}{t_1}\binom{v_2(2)}{t_2}...\binom{v_2(n)}{t_n}.
\end{equation*}
\end{corollary}
\begin{proof}
Proceeding in the same way as in Corollary~\ref{th2}, we complete the proof.
\end{proof}

\begin{remarks}
According to Corollary~\ref{thm4}, the partitions of $8$ in which each part $k$ has the multiplicity at most $v_2(k)$ are
$$8,\quad 6+2,\quad 4+4.$$
\begin{align*}
p_e(8)
&=\binom{v_2(8)}{1}+\binom{v_2(6)}{1}\binom{v_2(2)}{1}+\binom{v_2(4)}{2}\\
&=\binom{3}{1}+\binom{1}{1}\binom{1}{1}+\binom{2}{2}=5.
\end{align*}\end{remarks}

\section*{Acknowledgements} The first author acknowledges the financial support received from UGC, India through National Fellowship for Scheduled Castes  Students (NFSC) under grant Ref. no.: 211610029643.
 
 \section*{\bf Declarations}
 \noindent{\bf Conflict of interest.} The authors declare that there is no conflict of interest regarding the publication of
 this article.
 
 \noindent{\bf Human and animal rights.} The authors declare that there is no research involving human participants or
 animals in the contained of this paper.	
 
 \noindent{\bf Data availability statements.} Data sharing not applicable to this article as no data sets were generated or analysed during the current study.

\end{document}